\documentclass{article}

\usepackage[english]{babel}

\usepackage[a4paper,top=2cm,bottom=2cm,left=3cm,right=3cm,marginparwidth=1.75cm]{geometry}

\usepackage{amsmath, amsthm, amssymb}
\usepackage{graphicx}
\usepackage[colorlinks=true, allcolors=blue]{hyperref}
\usepackage{verbatim}
\usepackage{mathtools, mathrsfs}
\usepackage{enumitem}
\usepackage{cleveref}
\usepackage[normalem]{ulem}
\usepackage{svg}
\usepackage[numbers]{natbib}

\theoremstyle{definition}
\newtheorem{definition}{Definition}
\newtheorem{theorem}{Theorem}
\newtheorem{proposition}{Proposition}
\newtheorem{lemma}{Lemma}

\theoremstyle{remark}
\newtheorem{corollary}{Corollary}
\newtheorem{remark}{Remark}
\newtheorem{example}{Example}

\newcommand{\R}{\mathbb{R}}
\newcommand{\G}{\mathcal{G}}
\newcommand{\hl}{h\ell}

\DeclareMathOperator{\bt}{bt}

\title{Properties of leveled spatial graphs}
\author{S. Barthel, F. Buccoliero}
\begin{document}
\maketitle
\begin{comment}
    \begin{abstract}
  Leveled spatial graphs are studied regarding other spatial graph properties, including being free or paneled. New graph invariants are defined for an abstract graph permitting a leveled embedding: the level number and the Hamiltonian level number for Hamiltonian graphs. The level numbers for complete graphs and complete bipartite graphs are obtained.
\end{abstract}
\end{comment}
%\section{Introduction}
\begin{abstract}
We investigate the property of a spatial graph of having a leveled embedding and characterize the abstract graphs with this property. We show that all leveled embeddings are free and we compare leveled and paneled (also known as flat) embeddings.

Two new graph invariants are introduced: the level number, an invariant for graphs that admit a leveled embedding, and the Hamiltonian level number, an invariant for Hamiltonian graphs. These invariants provide a measure on how far a graph is from being planar. We study the relation between the (Hamiltonian) level number and other graph invariants that minimize the decomposition of a graph in planar subgraphs, namely the thickness and the book thickness of a graph. 

We characterize graphs with low level number and determine both the level number and the Hamiltonian level number of complete graphs and of complete bipartite graphs.
\end{abstract}
 
\section{Introduction}
The family of leveled spatial graphs was originally introduced in order to study the possibility of embedding a spatial graph cellular on a surface~\cite{LeveledCellularEmbeddings}. Intuitively, a leveled spatial graph consists of an unknotted cycle with subgraphs attached to it, arranged so that the subgraphs lie on stacked disks whose common boundary is the cycle. %In this paper, we study in which cases a spatial graph admits a leveled embedding. The definitions of the terms used in this section are given in \cref{section: preliminaries}.
%In the present paper, we study in which cases an abstract graph admits a leveled embedding.
Definitions are given in \cref{section: preliminaries}.

We give a new definition of leveled spatial graph and show that it is equivalent to the previous definition (\cref{prop: definitions of leveled are equivalent}).
With the new definition it becomes easier to characterize which abstract graphs, i.e. graphs without a specified embedding, admit a leveled embedding. This is done in \cref{prop: equivalent condition for having a leveled embedding} and \cref{cor: Hamiltonian graphs have a leveled embedding} where we show that every hamiltonian graph can be embedded as a leveled spatial graph.
Moreover, the new definition helps in comparing leveledness with other important spatial graph properties that relate to algebraic properties of the embedding: freeness and paneledness. In \cref{prop: leveled implies free} we show that all leveled spatial graphs are free. 

It is known that the family of paneled spatial graphs is a subfamily of free graph embeddings, since Robertson, Seymour and Thomas proved that a spatial graph is paneled if and only if all its subgraphs are free (\cite{RobertsonSeymourPaneled}). We compare paneledness and leveldness  and show in \cref{thm: paneled and 2connected implies leveled} that if a paneled graph embedding admits a spine, i.e. a cycle such that every fragment is a planar subgraph, then the embedding is leveled. A consequence of this result is \cref{cor: Hamiltonian graphs have a leveled embedding} which states that all Hamiltonian paneled spatial graphs are leveled.

Leveled spatial graphs can be viewed as a generalization of planar graphs, in the sense that a leveled spatial graph consists of stacked planar subgraphs whose union is the whole graph. Because of this interpretation, we define two new graph invariants: the level number and the Hamiltonian level number. These invariants are defined as the lowest number of levels a spatial graph (respectively, a Hamiltonian graph) admits ranging over all of its possible leveled embeddings (respectively, Hamiltonian leveled embeddings). The (Hamiltonian) level number measures how far a graph is from being planar, hence we compare the new invariants with two other invariants which also minimize the decomposition of a graph in planar subgraphs: the thickness~\cite{TutteThicknessofGraphs} and the book thickness~\cite{BookThickness} of a graph. Even though level number, thickness and book thickness are similar, we show in \cref{prop: inequalities concerning thickness level number and book thickness} and the discussion thereafter that they are all distinct.
Finally, we determine the level numbers of complete graphs and of complete bipartite graphs in \cref{prop: level number of Kn} and \cref{thm: level number of Knn}.

\section{Preliminaries}\label{section: preliminaries}
For the concepts not defined here, we refer the reader to books of topological graph theory and of general topology, for example \cite{GrossTuckerBook} and \cite{Munkres}. All graphs in this paper are connected undirected simple finite graphs. %The terms not defined here can be found for example in \cite{MR1367739}.

The \emph{fragments} (or bridges) of a graph~$G$ with respect to a cycle~$C$ of $G$ are the closures of the connected components of $G - C$~\cite{MR81471}. Two fragments are \emph{conflicting} if they either have pairs of endpoints on $C$ that alternate, or have at least three endpoints in common. 
The \emph{conflict graph} of a graph~$G$ with respect to a cycle~$C$ is the graph that has a vertex for each fragment of $G$ with respect to $C$, and two vertices are adjacent if and only if the corresponding fragments conflict.

A connected graph is \emph{Hamiltonian} if it contains a cycle that visits every vertex exactly once. Consequently, all fragments of a Hamiltonian graph with respect to the Hamiltonian cycle are closures of edges or loops. An abstract graph~$G$ is \emph{planar} if it has an embedding on a sphere or a plane. It is \emph{maximal planar} if $G$ is planar but adding any edge to $G$ makes it non-planar. $G$ is \emph{outerplanar} if it can be embedded on the sphere such that all vertices lie on the boundary of one region. The \emph{chromatic number} of a graph~$G$ is the minimum number of colors that are required to color every vertex of $G$ such that adjacent vertices have different colors. 

A \emph{graph embedding} is an embedding $f : G \rightarrow S^3$ of a graph~$G$ in the three-dimensional sphere $S^3$, up to ambient isotopy\footnote{For a formal definition of ambient isotopy we refer the reader to \cite{Munkres}. Informally, an ambient isotopy of a spatial graph allows to continuously deform edges and vertices of the graph along with the surrounding space $S^3$. In this way, self-intersections of the graph and shrinking knots to a point are avoided.}. The corresponding \emph{spatial
graph $\G$} is the image of this embedding up to ambient isotopy. A spatial graph is \emph{trivial} if it embeds in the 2-dimensional sphere~$S^2$.

For $n\ge 1$, an \emph{$n$-book} consists of a line $L$, called the spine of the book, and $n$ distinct half-planes, the \emph{pages} of the book, with boundary $L$. An \emph{$n$-book embedding of a graph $G$} is an embedding of $G$ in an $n$-book such that each vertex of $G$ lies on the spine $L$ and each edge lies in the interior of one page of the book. The vertices of $G$ occur in some order along the spine of the book and this sequence is called the \emph{printing cycle} of the embedding~\cite{BookThickness}.

A spatial graph $\G$ is \emph{paneled} (or \emph{flat}) if every cycle $C$ of $\G$ bounds an embedded open disk in $S^3$ disjoint from $\G$. A spatial graph~$\G$ is \emph{free} if the fundamental group of its complement $\pi_1(S^3 \setminus \G)$ is a free group, i.e. a group whose generators do not have any relations. 

A \emph{diagram} of a spatial graph~$\G$ is the image of a projection of $\G$ onto the Euclidean plane~$\mathbb{R}^2$ such that all intersections in the image are transversal, no vertices and at most two points are mapped to a crossing, and each double point is assigned as an over- or under-crossing. A diagram is \emph{reduced} if no Reidemeister moves of type 1 and 2 can be applied to reduce its number of crossings. These moves are shown in \Cref{fig: reidemeister moves}. A cycle is \emph{unknotted} if it can be embedded as the equator of a sphere~$S^2$.

In our previous work~\cite{LeveledCellularEmbeddings} we defined leveled spatial graphs as follows:
\begin{definition}\label{def:leveled_previous}
 A \emph{leveled spatial graph~$\G$} is a connected spatial graph that contains an unknotted cycle~$C$ called its \emph{spine} such that each fragment~$f$ of $\G$ with respect to $C$ can be embedded in a disk~$D_f$ whose boundary is identified with $C$, with interior $\mathring{D_f}$ disjoint from $\G - f$, in a way that no two disks $D_f$ and $D_g$ for distinct fragments $f$ and $g$ intersect.    
\end{definition}
For the special case of the spine to be a Hamiltonian cycle of the graph, we introduce Hamiltonian leveled spatial graphs:
\begin{definition}\label{def: hamiltonian_leveled}
    A \emph{Hamiltonian leveled spatial graph} $\mathcal{H}$ is a leveled spatial graph whose spine is a Hamiltonian cycle of the underlying abstract graph.
\end{definition}

\begin{figure}
    \centering
    \includegraphics[width=0.5\linewidth]{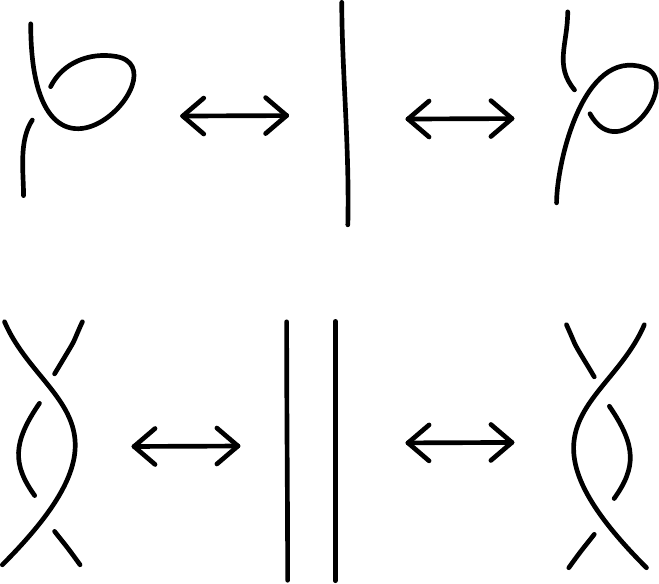}
    \caption{Reidemeister moves of type 1 (above) and type 2 (below).}
    \label{fig: reidemeister moves}
\end{figure}

In the present paper, we define leveled spatial graphs alternatively (\cref{def: leveled embedding}) as the spatial graphs with properties defined in \cref{def: semi-leveled} and \cref{def: divided in levels}. This facilitates the study of leveledness in relation to other spatial graph properties. We prove the equivalence of the two definitions in~\cref{prop: definitions of leveled are equivalent}. 

\begin{definition}\label{def: semi-leveled}
    Let $G$ be a graph. We say that an embedding of $G$ is \emph{semi-leveled} if there exists an unknotted cycle~$C$ in $G$, called the \emph{spine}, and a diagram~$D$ of the embedded graph in which no fragment~$f$ with respect to $C$ has self-crossings, crossings with $C$, or both over- and under-crossings with another fragment. The diagram~$D$ is called a~\emph{semi-leveled diagram} of the semi-leveled spatial graph~$\G$. We say that $\G$ is \emph{semi-leveled with respect to spine $C$} when we want to stress the choice of the spine of the semi-leveled embedding.
\end{definition}

\begin{definition}\label{def: divided in levels}
    Let $\G$ be a semi-leveled spatial graph and let $D$ be a reduced semi-leveled diagram of $\G$, such that every fragment with respect to the spine~$C$ lies in the bounded connected component of $\mathbb{R}^2\setminus C$. We say that $D$ is \emph{divided in levels} if there exists a partition $\mathscr{L}= \{L_1, L_2, \dots, L_n\}$ of all the fragments, called the \emph{level partition}, such that the following holds: $L_1$ is the set of fragments that do not cross over any other fragment, and for $i>1$, the set $L_i$ consists of fragments that cross over at least one fragment $f\in L_{i-1}$ and only cross over fragments belonging to $L_k$, with $k<i$. The fragments belonging to $L_i$ are called \emph{fragments of level~$i$} and $n$ is the \emph{number of levels} of the embedding of $\G$. 
\end{definition}
Note that the requirement that every fragment with respect to the spine $C$ lies in the bounded connected component of $\mathbb{R}^2 \setminus C$ can be obtained without loss of generality since the fragments do not cross the spine $C$ and so they can be considered lying in the bounded component of $\mathbb{R}^2\setminus C$ up to ambient isotopy.
\begin{definition}\label{def: leveled embedding}
    A connected spatial graph~$\G$ has a \emph{leveled} embedding if it has a semi-leveled embedding with a semi-leveled diagram~$D$ that is divided in levels.~(\Cref{fig: example leveled embedding})
\end{definition}

\begin{figure}[h]
    \centering
    \includegraphics[width=0.5\linewidth]{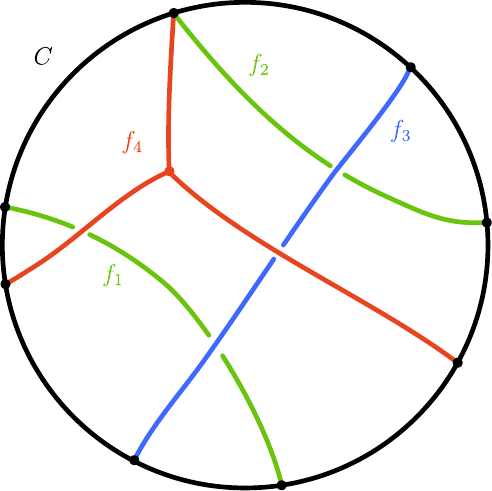}
    \caption{A leveled spatial graph with respect to spine $C$ and three levels. The fragments $f_1$ and $f_2$ are fragments of level 1, while $f_3$ belongs to $L_2$ and $f_4$ to $L_3$.}
    \label{fig: example leveled embedding}
\end{figure}

\begin{remark}
  Clearly, a leveled embedding is semi-leveled but the opposite does not hold true, i.e. there exist semi-leveled diagrams which cannot be divided in levels. Indeed, the spatial graph~$\G$ containing the torus knot~$T(3,2)$ shown in \Cref{fig: spatial graph with torus knot not leveled but semileveled} is semi-leveled but not leveled. Indeed, inspecting the six Hamiltonian cycles of $\G$ as possible spines, we can see that it is either not possible to find a level partition~$\mathscr{L}$ or that there is no semi-leveled diagram.
\end{remark}
\begin{figure}
    \centering
    \includegraphics[width=0.4\textwidth]{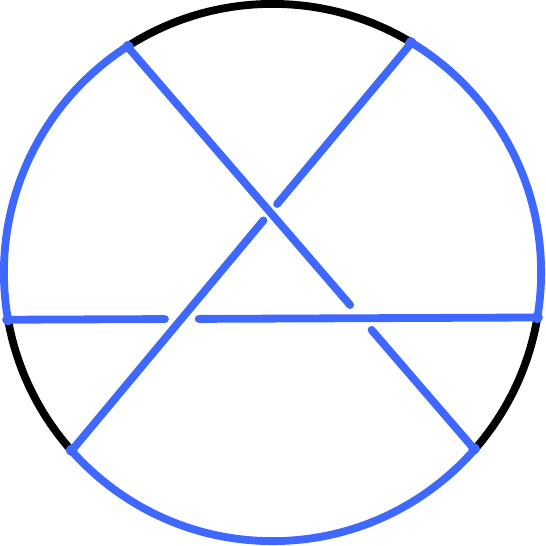}
    \caption{Semi-leveled embedding of a spatial graph containing the torus knot~$T(3,2)$, highlighted in blue. This spatial graph is not leveled.}
    \label{fig: spatial graph with torus knot not leveled but semileveled}
\end{figure}

\begin{proposition}\label{prop: definitions of leveled are equivalent}
    Definition~\ref{def:leveled_previous}  and \cref{def: leveled embedding} are equivalent.
\end{proposition}
\begin{proof}
    If $\G$ has a semi-leveled embedding with semi-leveled diagram~$D$, then every fragment~$f$ with respect to the spine~$C$ is a planar subgraph of $\G$ without self-crossings or crossings with $C$. Therefore, $f$ can be embedded on a disk $D_f$ whose boundary is identified with $C$. Since $D$ is divided in levels, the interior of $D_f$ is disjoint from $\G - f$ and any other $D_g$, for $g$ another fragment of $\G$. 
    For the other implication, order the disks in a sequence $D_1 , \dots, D_d$ such that disks~$D_f$ that together bound a ball in the complement of the graph are labeled consecutively modulo the number of disks~$d$. Consider the diagram~$D$ given by embedding the spine~$C$ on a plane $\mathbb{R}^2$ and projecting the disks $D_f$ to the bounded connected component of $\mathbb{R}^2 \setminus C$. Without loss of generality, the diagram can be assumed to be reduced, and such that the fragment of disk~$D_1$ is not crossing over any other fragment. The diagram~$D$ is semi-leveled, because no fragment has self-crossings, crossings with $C$, or both over- and under-crossings with another fragment by construction. Define a level partition in the following way: Label the fragments of $\G$ by the index of the disks they belong to. 
   Assign $f_1$ and all fragments that do not cross over any other fragment in $D$ to level~1. %Otherwise, being $f_{i_1} \in L_{j_1}, \dots, f_{i_h} \in L_{j_h}$ the fragments conflicting with $f_i$ with index smaller than $i$, assign $f_i$ to $L_m$, where $m= 1+\max(j_1, \dots, j_h).$ 
    Assign to level~2 the fragment with smallest label that conflicts with a fragment assigned to level 1 and does not conflict with fragments assigned to level 2. Continue in the same way until all fragments are assigned to a level.
    The defined partition is a level partition and therefore $D$ is divided in levels.%   Moreover, they define a level partition by considering fragments $f$ and $g$ on the same level if they do not conflict with each other and there is no other disk in between $D_f$ and $D_g$.
\end{proof}

\section{Characterization and stability of leveled embeddings}
In \cref{prop: equivalent condition for having a leveled embedding} we characterize the abstract graphs with a leveled embedding.
\begin{proposition}\label{prop: equivalent condition for having a leveled embedding}
   % For every abstract graph~$G$ which has a cycle~$C$ such that each fragment with respect to $C$ is a planar subgraph, there exists an embedding $\G$ of $G$ in $S^3$ such that $\G := \G(G)$ is a leveled embedding. Moreover, if a graph $G$ has a leveled embedding, then it has a cycle $C$ such that each fragment with respect to $C$ is a planar subgraph by definition. In other words, a
    An abstract graph has a leveled embedding if and only if it has a cycle for which each fragment is planar.
\end{proposition}
\begin{proof}
It is clear by definition that a graph that admits a leveled embedding contains a cycle for which each fragment is planar. For the other implication, let $f_1, \dots, f_n$ be the planar fragments of $G$ with respect to $C$. Embed $C$ unknotted and choose it as the spine. Embed $f_i$ for each $i$ on a disk $D_{f_i}$, which is possible because $f_i$ is a planar graph. Identify the boundary of $D_{f_i}$ with $C$ for every $i$ in such a way that the interior of $D_{f_i}$ does not intersect the interior of $D_{f_j}$ for any $j\neq i$ which is possible by stacking the disks above each other.
\end{proof}
\begin{corollary}\label{cor: Hamiltonian graphs have a leveled embedding}
    Every Hamiltonian graph has a leveled embedding.
\end{corollary}

It follows from \cref{def: leveled embedding}, that an abstract graph that does not contain a cycle such that every fragment is a planar subgraph, cannot be embedded semi-leveled.
\begin{figure}
    \centering
    \includegraphics[width=0.8\textwidth]{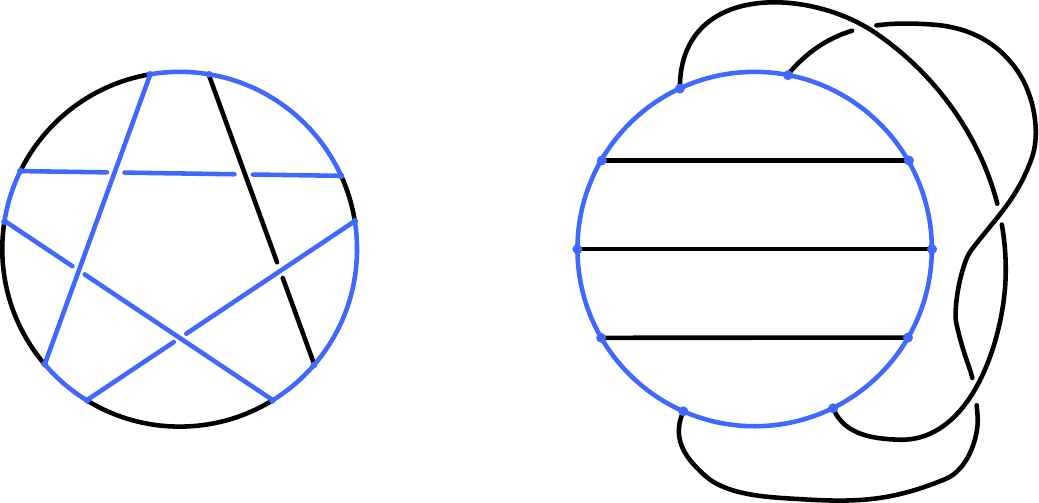}
    \caption{Left: a Hamiltonian leveled spatial graph with a selected cycle in blue. Right: the embedding with the selected cycle as spine is not semi-leveled and hence leveled because two fragments cross over and under each other.}
    \label{fig: changing spine to knotted Hamiltonian makes it non leveled}
\end{figure}

Leveled embeddings are not ``stable" in the sense that not all cycles of the graph can be chosen as spines, not even in the case that all fragments with respect to the spine are trivial.
%for a spatial graph with leveled embedding $\G$, changing the spine does not guarantee that the new embedding with respect to the new spine is again leveled.
For example, on the left of \Cref{fig: changing spine to knotted Hamiltonian makes it non leveled} a Hamiltonian leveled spatial graph is shown. 
The blue Hamiltonian cycle cannot act as a spine since all diagrams in which the blue Hamiltonian does not have self-crossings, two of its fragments have both over- and under-crossings.
 
%If the blue Hamiltonian cycle is chosen as new spine, we obtain the spatial graph on the right, which is not leveled as two fragments cross over and under each other.

The problem in \Cref{fig: changing spine to knotted Hamiltonian makes it non leveled} is that the spatial graph contains a knot.
%which makes the embedding with the new spine not leveled.
However, we can still find examples of knotless leveled spatial graphs with a cycle whose all fragments are trivial, such that the cycle cannot act as spine, see~\Cref{fig: changing spine to knotless Hamiltonian makes it non leveled}.
It is known that changing the spine of a leveled embedding affects its number of levels \cite{LeveledCellularEmbeddings}. 
\begin{figure}[h]
    \centering
    \includegraphics[width=0.8\textwidth]{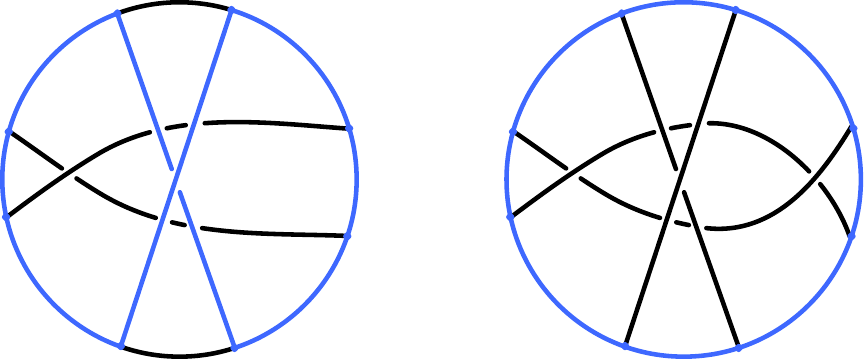}
    \caption{Left: a knotless Hamiltonian leveled spatial graph with a selected cycle in blue. Right: the selected cycle cannot be chosen as a spine because it forces two fragments to cross over and under each other.}
    \label{fig: changing spine to knotless Hamiltonian makes it non leveled}
\end{figure}

\section{Leveled spatial graphs are free}
In \cref{prop: leveled implies free} we prove that the family of leveled embeddings is contained in the family of free embeddings. To prove this result, we need first to show a sufficient condition for a spatial graph to be free.
A \emph{bouquet} is a graph consisting of a single vertex. The graph $G/e$, where $e$ is an edge of $G$ with distinct endpoints, is obtained form $G$ by contracting the edge $e$ and letting its endpoints coincide. In a similar fashion, the graph $G/T$, where $T$ is a spanning tree of $G$, is obtained from $G$ contracting each edge of $T$ one at a time.
\begin{lemma}\label{lem: if bouquet is trivial then free}
    Let $\G$ be a spatial graph. If there exists a spanning tree~$T$ such that the bouquet $\G/T$ that is obtained by contracting $T$ is a trivial spatial graph, then $\G$ is free.
\end{lemma}
\begin{proof}
    If $\G/T$ is trivial, then it is free. Since the contraction of $T$ can be reversed by subdivision of edges and splittings of vertices and these do not change the fundamental group of the complement of a neighborhood of the graph, $\G$ is also free.
\end{proof}
The converse is not true. Indeed, consider the trefoil knot with one unknotting tunnel drawn in~\Cref{fig: torus knot with unknotting tunnel}. This is a free spatial graph. Since the contraction of edges does not change the complement of the spatial graph, contracting an edge which is part of the knot results in a free spatial graph that is not trivial. However, the only spanning tree of the obtained graph is its vertex and therefore no spanning tree can be found that could be contracted to obtain a trivial graph.
%A converse that works is the following.

\begin{figure}[h]
    \centering
    \includegraphics[width=0.8\textwidth]{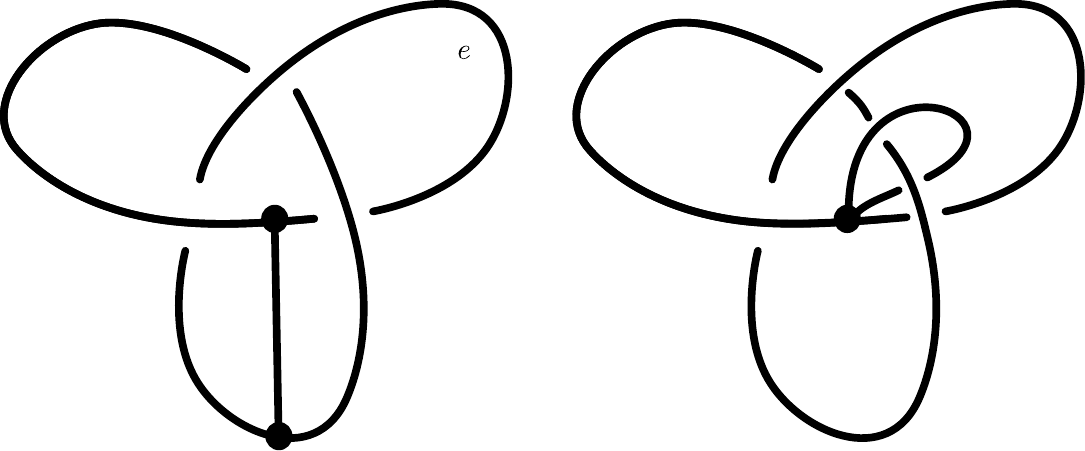}
    \caption{Left: the torus knot $T(3,2)$ with an unknotting tunnel is a free spatial graph. Right: after contracting the edge $e$, the spatial graph remains free.}
    \label{fig: torus knot with unknotting tunnel}
\end{figure}
However, Makino and Suzuki showed that there exists a sequence of edge contractions and vertex splittings between any pair of spatial graphs that have the same neighbourhood~\cite{makino1995}, where the neighbourhood of a spatial graph $\G$ is a tubular neighbourhood of its edges glued to tubular neighbourhoods of its vertices in such a way that $\G$ is embedded in the neighbourhood. This allows us to state the following:
\begin{remark}
    If a spatial graph~$\G$ is free, then there exists a sequence of vertex splittings and edge contractions producing a spatial graph~$\tilde{\G}$ such that $\tilde{\G}$ has a spanning tree~$T$ such that $\tilde{\G}/T$ is trivial.
\end{remark}
\begin{proposition}\label{prop: leveled implies free}
    If $\G$ is a leveled spatial graph, then $\G$ is free.
\end{proposition}
\begin{proof} 
    Given that $\G$ is leveled, we can embed the fragment $f$ of $\G$ in a disk $D_f$ such that $\partial D_f$ coincides with $C$ and the interior of $D_f$ does not intersect $\G$ except in $f$. If the fragment~$f$ contains a vertex of $\G$ that is not a vertex of the spine, we can contract edges of $f$ inside $D_f$ until there are no vertices of $f$ outside the spine. As the obtained spatial graph~$\G'$ is Hamiltonian leveled, we can again embed each fragment~$f'$ of $\G'$ on a disk~$D_f'$ disjoint from $\G' - f'$. 
    Therefore, we can assume $\G$ to be Hamiltonian leveled without loss of generality. Note that for a Hamiltonian leveled graph~$\G$, $f$ splits $D_f$ into two disks, $D_{f, 1}$ and $D_{f,2}$. Moreover, since $C$ is unknotted, there exists a disk~$D_C$ such that $\partial D_C$ coincides with $C$ and the interior of $D_C$ does not intersect $\G$.
    By~\cref{lem: if bouquet is trivial then free}, it is enough to find a spanning tree~$T$ such that~$\G/T$ is trivial to prove the statement. Let $C$ be the spine of $\G$ and let $e$ be any edge in $C$.
    Consider $T = C - e$, then $T$ is a spanning tree for $\G$.
    Contracting $T$, the boundary of $D_f$ becomes the edge $e$, while the fragment $f$ becomes a loop based at the only vertex of the bouquet $\G/T$. Only one of $D_{f,1}, D_{f,2}$ remains a disk, without loss of generality we can assume it is $D_{f,1}$ for every fragment $f$, while the other becomes a pinched disk. In a similar way, $D_C$ remains a disk with boundary given by $e$. As the collection of disks $D_C, D_{f,1}$ for each fragment $f$ of $\G$ have interiors which are pairwise disjoint, the spatial graph $\G/T$ is paneled. By (2.2) of \cite{RobertsonSeymourPaneled}, $\G/T$ is trivial since it is a paneled embedding of a planar abstract graph. Therefore, $\G$ is free.
\end{proof}

\section{Relation between leveled and paneled spatial graphs}
In this section we compare leveled and paneled spatial graphs. We find that neither property implies the other, but that they are closely related as specified in \cref{thm: paneled and 2connected implies leveled}.

There are many leveled embeddings that are not paneled, given that
a leveled embedding containing a knot is certainly not paneled. As any knot is contained in the leveled embedding that is obtained by considering a book representation of the knot and adding the edges of the spine of the book to the knot, there exists an infinite family of leveled embeddings which are not paneled. 

A paneled embedding of a spatial graph $\G$ is not necessarily semi-leveled. 
Indeed, consider two copies of $K_{3,3}$ connected by an edge, as in \Cref{fig: connected K33}. The embedding is paneled, because $K_{3,3}$ is paneled, as shown by Robertson, Seymour and Thomas (\cite{RobertsonSeymourPaneled}). However, the embedding is not semi-leveled, because at least one fragment with respect to any cycle is non-planar.
\begin{figure}
    \centering
    \includegraphics[width=0.8\textwidth]{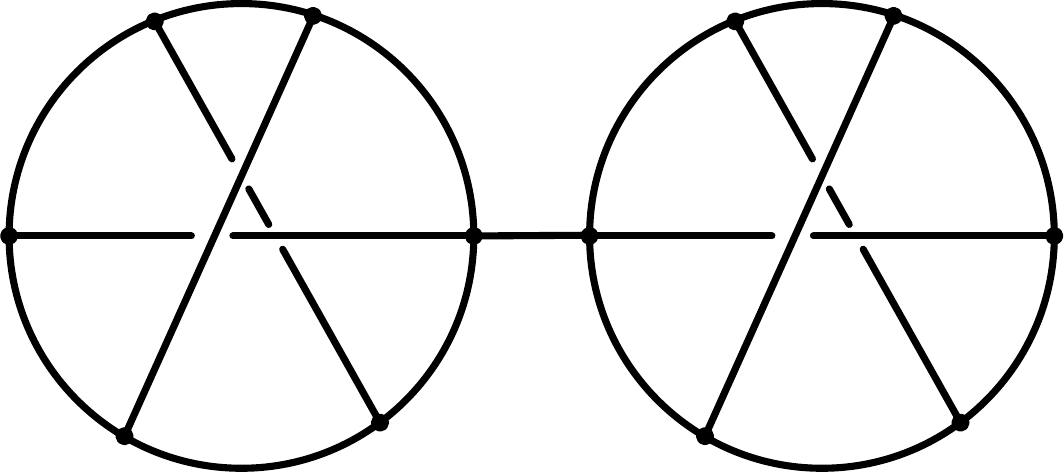}
    \caption{Two copies of $K_{3,3}$ connected by an edge have a paneled embedding but no semi-leveled embedding.}
    \label{fig: connected K33}
\end{figure}

Nonetheless the two properties combined are enough to imply leveledness.
    
\begin{theorem}\label{thm: semileveled and paneled implies leveled}
    If the embedding of a spatial graph $\G$ is semi-leveled and paneled, then it is leveled.
\end{theorem}
\begin{proof}
    We will prove that if a spatial graph $\G$ is semi-leveled for a reduced diagram $D$, and $D$ is not divided in levels, then $\G$ is not paneled.
            
    We can suppose without loss of generality that $D$ is reduced and that all fragments with respect to the spine $C$ lie in the bounded connected component of $\R^2 \setminus C$, because we can apply Reidemeister moves locally away from $C$ and fragments do not cross $C$.
            
    A semi-leveled diagram $D$ is not divided in levels if and only if a level partition $\mathscr{L}$ does not exist. 
    This is the case if and only if there exists a fragment $f$ such that $f \notin L_i$ for every $i$. Note that, even though $\mathscr{L}$ does not exist, $L_i$ can still be defined for every $i$. This means that $f$ crosses over another fragment $g$ that itself does not belong to $L_i$ for every $i$. Indeed, $f$ must cross another fragment $g$ since otherwise $f$ could be assigned to $L_1$. Furthermore, $g \notin L_i$ for every $i$, since if all fragments that are overcrossed by $f$ belong to one of $L_{j_1}, \dots, L_{j_k}$, then $f$ would belong to $L_{\max(j_1, \dots, j_k)+1}$, and given that $f$ does not belong to any $L_i$, neither does $g$. By the same argument, $g$ crosses over another fragment $h$ which does not belong to $L_i$ for every $i$. Since the number of fragments of $\G$ is finite, there exists a sequence of fragments $f_1, \dots, f_t$ such that $f_1 <f_2<\dots<f_t<f_1$; where $f_a<f_b$ if the fragment $f_a$ crosses under fragment $f_b$ non-trivially (meaning that no Reidemeister moves of type~2 can be applied to remove the crossing between $f_a$ and $f_b$).
            
    Note that the sequence of fragments must contain at least three fragments because there cannot be two fragments which cross both over and under each other since $\G$ is semi-leveled.
            
    If the fragment $f_i$ in the sequence $f_1, \dots, f_t$ crosses under the fragment $f_k$, with $k >i+1$, then we can reduce the sequence by removing all fragments from $f_{i+1}$ to $f_{k-1}$. In this way, we obtain a sequence $f_1 < f_2 < \dots < f_t < f_1$ where $f_i$ does not cross under any fragment other than $f_{i+1}$. In the same way, $f_t$ crosses under only $f_1$. This also implies that $f_i$ crosses over only $f_{i-1}$ for $i=2, \dots, t$ and $f_1$ crosses over only $f_t$. Given that fragments are connected, we can restrict the sequence of fragments to sequence of edges $e_1, \dots, e_t$ with $e_1< \dots < e_t<e_1$, where $e_i$ crosses under only $e_{i+1}$ (modulo $t$) and where $e_a<e_b$ if the edge $e_a$ crosses under edge $e_b$ non-trivially (meaning that no Reidemeister moves of type 2 can be applied to remove the crossing between $e_a$ and $e_b$).

    We now prove that the subdiagram $\tilde{D}\subset D$ that is formed by the spine of $D$ and the edges of the sequence  $e_1, \dots, e_t $ from above depicts a nontrivial knot or link for $t\ge 4$, implying that $\G$ cannot be paneled. We deal with the case $t=3$ later.
    
    Suppose that the sequence $e_1< \dots < e_t<e_1$ contains at least 4 edges, i.e. $t \ge 4$. Consider edge $e_i$. The labeling is always modulo $t$ in the following. We know that $e_i$ only has crossings with the edges $e_{i+1}$ and $e_{i-1}$. We claim that $e_i$ crosses exactly once over $e_{i-1}$. 
    
    Suppose that $e_{i-1}$ crosses under $e_i$ at least twice. This means that there exists at least one region in the diagram $\tilde{D}$ formed by a superarc of $e_i$ and a subarc of $e_{i-1}$ (two such regions are highlighted with a red circle in \Cref{fig: 2edges 3crossings}). In order for the diagram to be reduced, another edge~$e$ would need to cross this region such that $e$ crosses over $e_{i}$ and under $e_{i-1}$, since otherwise a Reidemeister move of type~2 can be performed to reduce the crossings between $e_i$ and $e_{i-1}$. Since $e_{i-1}$ only has crossings with $e_{i}$ and $e_{i-2}$, $e$ must equal $e_{i-2}$. Also, since $e_{i}$ only has crossings with $e_{i-1}$ and $e_{i+1}$, $e$ must equal $e_{i+1}$. Since $t\ge 4$, $e_{i-2} \neq e_{i+1}$ and no edge~$e$ exists. Therefore, a Reidemeister move of type~2 can be performed to remove the region. Applying the same reasoning to every region formed between $e_i$ and $e_{i-1}$ concludes that these fragments cross each other exactly once.

    Thus, if $t \ge 4$, $e_i$ crosses under $e_{i+1}$ exactly once. Up to mirror image, this means that we obtain an embedding as in \Cref{fig: toruslike shape}. In such an embedding, there exists a torus link or a torus knot $T(t,2)$. Indeed, consider the cycle obtained by starting at edge $e_1$, running through edge $e_3$, $e_5$ and so on, as depicted in blue in \Cref{fig: toruslike shape} for $6$ edges. If the number $t$ of edges is odd, this process ends with the torus knot $T(t,2)$. If $t$ is even, we obtain the torus link $T(t,2)$ by starting another cycle at edge $e_2$ and running through even-indexed edges. In both cases, the spatial graph contains a link or a knot and hence it cannot be paneled.

         \begin{figure}[h!]
        \centering
        \includegraphics[width=0.4\textwidth]{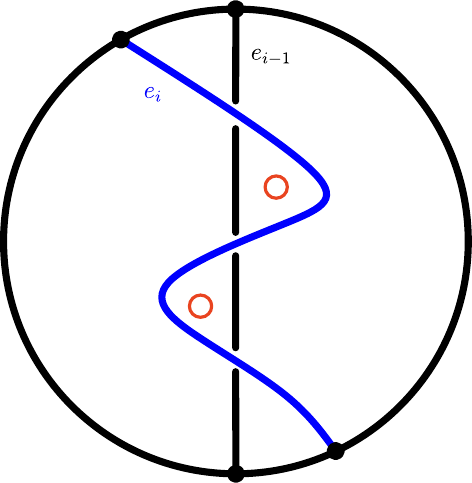}
        \caption{The crossings between $e_i$ and $e_{i-1}$ can be reduced to a single crossing, because $e_{i-2}$ does not cross $e_i$ and does not cross both over and under $e_{i-1}$.}
        \label{fig: 2edges 3crossings}
    \end{figure}
\begin{comment}
    \begin{figure}
        \centering
        \includegraphics[width=0.4\linewidth]{BraidLike.pdf}
        \caption{If the sequence $e_1, \dots, e_t$ is formed by only three fragments, any two fragments can cross each other more than once.}
        \label{fig: braid-like tangle}
    \end{figure}
\end{comment}
    \begin{figure}[h!]
        \centering
        \includegraphics[width=0.4\textwidth]{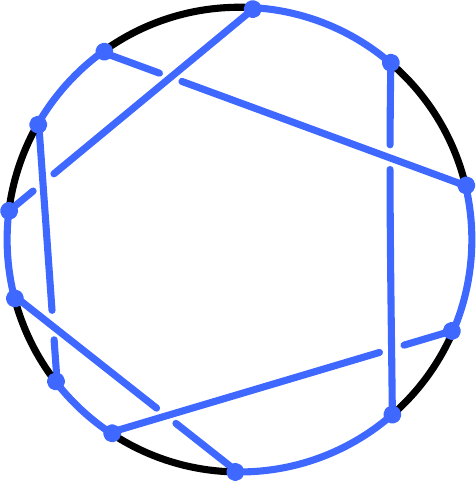}
        \caption{The only embedding (up to mirror image) if each edge crosses the following once for $t=6$. The embedding contains a torus link, highlighted in blue.}
        \label{fig: toruslike shape}
    \end{figure}
    
    We are left with the case $t=3$. For $e_1 < e_2 < e_3 < e_1$, the graph~$\tilde{\G} \subset \G$ that is determined by $\tilde{D}$ has 6 vertices and 9 edges. According to Kuratowski theorem, such a graph is not planar if and only if it is isomorphic to $K_{3,3}$~(\cite{kuratowski1930}). 
    Suppose that the underlying graph $\tilde{G}$ of $\tilde{\G}$ is not $K_{3,3}$ and therefore is planar. As $\tilde{\G}$ is not leveled due to the existence of the edge sequence, this implies that it cannot be trivial, as every trivial planar graph is leveled by considering the cycle in the boundary of any face as spine.  Since by Theorem~2.2 of \cite{RobertsonSeymourPaneled} an embedding of a planar graph is paneled if and only if it is trivial, it follows that $\G$ is not paneled. It remains to show that the underlying graph $\tilde{G}$ of $\tilde{\G}$ cannot be $K_{3,3}$. We do this by arguing that every semi-leveled and paneled embedding of $K_{3,3}$ is leveled: Theorem~(3.1) of \cite{RobertsonSeymourPaneled} states that there exist only two non-ambient-isotopic paneled embeddings of $K_{3,3}$, depicted in \Cref{fig: flat emb of K33}. Since both these embeddings of $K_{3,3}$ are leveled, it follows that every paneled and semi-leveled embedding of $K_{3,3}$ is leveled.
    \begin{figure}[h!]
        \centering
        \includegraphics[width=0.8\textwidth]{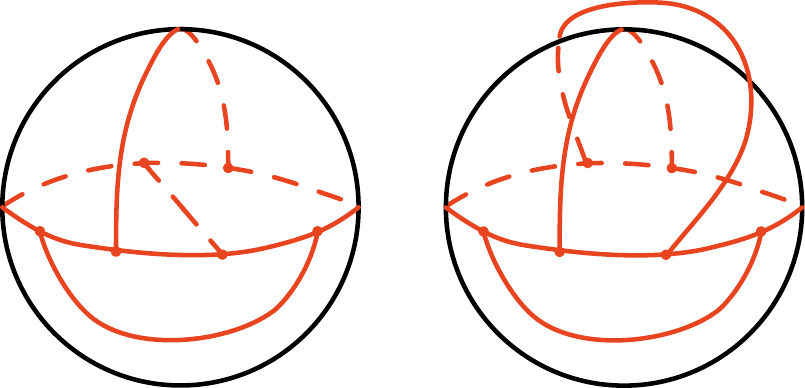}
        \caption{The two embeddings of $K_{3,3}$ that are paneled
        are also leveled.}
        \label{fig: flat emb of K33}
    \end{figure}
\end{proof}

\begin{theorem}\label{thm: paneled and 2connected implies leveled}
    If the embedding of a connected spatial graph $\G$ is paneled and there exists a cycle $C$ in $\G$ such that every subgraph $f_i \cup C$, where $f_i$ is a fragment with respect to $C$, is a trivial spatial subgraph, then the embedding of $\G$ is leveled.
\end{theorem}
To prove this result, we use the following lemma of B\"ohme (\cite{bohme1990}). A circuit is a closed path in a graph where edges cannot be repeated.
\begin{lemma}\label{lem: there are disjoints disks}
    Let $\varphi$ be a paneled embedding of a graph $G$ into $S^3$, and let $C_1,C_2,\dots,C_n$ be a family of circuits of $G$ such that for every $i \neq j$, the intersection of $C_i$ and $C_j$ is either connected or empty. Then there exist pairwise disjoint open disks $D_1, D_2,\dots , D_n$ disjoint from $\varphi(G)$ and such that $\varphi(C_i)$ is the boundary of $D_i$ for $i = 1,2,\dots,n$.
\end{lemma}

\begin{proof}[Proof of \cref{thm: paneled and 2connected implies leveled}]
    We will show that $\G$ is leveled with spine $C$, using the first definition of leveled spatial graph given in the text of \cref{section: preliminaries}.

For the fragment $f_i$ of $\G$ with respect to $C$, consider the subgraph $f_i\cup C$. By hypothesis, $f_i \cup C$ is a trivial spatial graph. Moreover, the spatial graph $f_i \cup C$ embeds cellular in the sphere $S^2$ since $f_i \cup C$ is connected. Let $F_{i,1}, \dots , F_{i,n}$ be the cells that are not bound by $C$, and denote their boundaries with $C_{i,k}= \partial F_{i,k}$. Note that $C$ is the boundary of two embedded disks: one disk is a cell of the cellular embedding, and the other disk is its complement $D_{F_{i}} = \bigcup_{k=1}^n F_{i,k}\cup C_{i,k}$.
Since $\G$ is paneled, there exist open disks $D_{i,1}, \dots , D_{i,n}$ disjoint from $\G$ with $\partial D_{i,k} = C_{i,k}$. Since the union $F_{i,k}\cup D_{i,k} \cup C_{i,k}$ is a sphere for every $k$, replacing the cell~$F_{i,k}$ of $D_{F_{i}}$ with the disk~$D_{i,k}$ produces a new disk $D_{f_i}= \bigcup_{k=1}^n D_{i,k}\cup C_{i,k}$ with $\partial D_{f_i} = C$. By construction, $D_{f_i} \cap \G = C \cup f_i$ and we are left to prove that $D_{f_i} \cap D_{f_j} = \emptyset$ for $i \neq j$. To apply \cref{lem: there are disjoints disks}, we need to show that $C_{i,k} \cap C_{j,l}$ is either empty or connected for all choices of $i,j,k,l,$ with $ i\neq j$. If $C_{i,k} \cap C_{j,l} \neq \emptyset$, the intersection is part of $C$ by construction, and we are left to consider cycles that run through $C$. The intersection of any such cycle with $C$ is connected, given that fragments are connected. It follows that the assumptions of \cref{lem: there are disjoints disks} are fulfilled and therefore $D_{i,k}$ and $D_{j,l}$ are disjoint for all $i,j,k,l,$ with $i\neq j$. Consequently, $D_{f_i} \cap D_{f_j} = C$ for $i \neq j$ and the embedding is shown to be leveled.
\end{proof}

The following corollary is implied by \cref{thm: paneled and 2connected implies leveled}, but we give an independent proof below.
\begin{corollary}\label{thm: paneled and Hamiltonian implies leveled}
    If the embedding of a Hamiltonian spatial graph $\G$ is paneled, then it is leveled.
\end{corollary}
\begin{proof}
   Due to \cref{thm: semileveled and paneled implies leveled}, it is enough to prove that paneled and Hamiltonian imply semi-leveled. That is, it is sufficient to prove that there exists a cycle $C$ such that no fragment with respect to $C$ crosses itself or $C$, and such that two fragments do not cross above and below each other.

    Consider a Hamiltonian cycle $C$ of $\G$. We show that $C$ can be chosen as spine for $\G$: Because $\G$ is paneled, $C$ is unknotted. Therefore, choose a diagram $D$ where $C$ is unknotted and every fragment with respect to $C$ projects into the bounded connected component of $\mathbb{R}^2 \setminus C$. This is possible, because no fragment intersects non-trivially with $C$ given that $\G$ is paneled. Let $f$ and $g$ be two fragments. We show that they do not cross over and under each other. Indeed, consider, without loss of generality, one of the two cycles formed by $f$ and part of the spine $C$. Call it $K$. As $\G$ is paneled, there exists a disk $D'$ such that $\partial D'= K$, and $\G$ (so in particular $g$) does not intersect the interior of $D'$. This means that, if $g$ crosses $f$ both over and under, $g$ does not cross the disk $D'$ and so we can perform an ambient isotopy to ensure that the crossings between $g$ and $f$ are all of the same type.

    Similarly, $f$ cannot cross itself in a non-trivial way, because $\G$ is paneled and hence knotless. Therefore, $C$ is the spine of a semi-leveled embedding of $\G$.
\end{proof}
Note that the proof of \cref{thm: paneled and Hamiltonian implies leveled} shows that every Hamiltonian cycle of a semi-leveled and paneled embedding can be chosen as spine.

\section{A new invariant to measure non-planarity}
A leveled embedding of a spatial graph $\G$ with respect to spine $C$ is endowed with an integer number: its number of levels. We can define a spatial graph invariant using the number of levels of a leveled embedding for $\G$.

As a spatial graph is defined up to ambient isotopy, also the \emph{level number} of a spatial graph should be defined up to ambient isotopy. Since a leveled embedding might have more than one cycle that can act as spine, and considering another spine of a spatial graph can result in a different number of levels (see for example~\Cref{fig: mobius ladder}, where two different choices of the spine for a Hamiltonian leveled spatial graph have different number of levels), we need to range over all possible spines for a leveled embedding of the spatial graph.

\begin{figure}[h]
    \centering
    \includegraphics[width=0.8\textwidth]{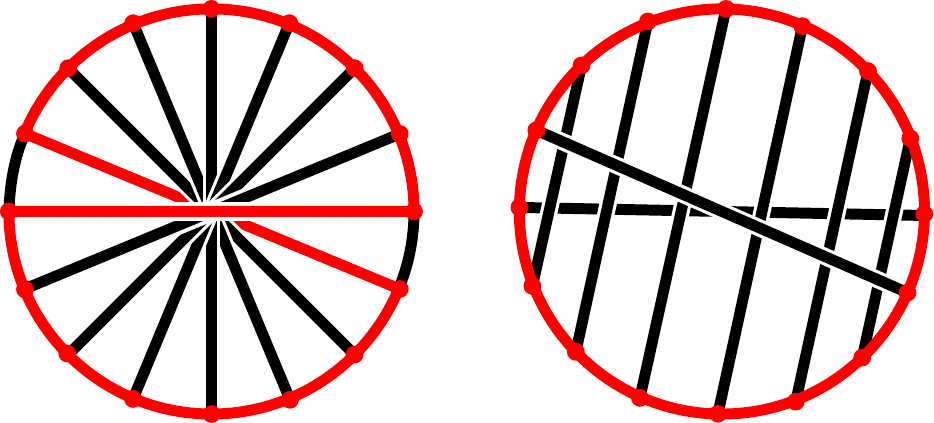}
    \caption{Left: a Hamiltonian leveled embedding of the M\"obius ladder on 16 vertices $M_{16}$ with eight levels and a selected cycle in red. Right: a Hamiltonian leveled embedding of $M_{16}$ with three levels, where the red cycle of the left figure is chosen as the spine.}
    \label{fig: mobius ladder}
\end{figure}

\begin{definition}\label{def: level number spatial graph}
    Given a spatial graph $\G$, we define its \emph{level number}, denoted $\ell(\G)$, to be the minimum number of levels of $\G$ ranging over all possible spines of $\G$. If $\G$ does not admit a leveled embedding, we set $\ell(\G)=\infty.$
\end{definition}
%As it is often the case in topological graph theory (\cite{BookThickness}, \cite{mellor2016}), a spatial graph invariant defines a related abstract graph invariant, 
Ranging over all possible embeddings of an abstract graph $G$ allows us to define a corresponding property of abstract graphs:
\begin{definition}
    Given an abstract graph $G$, we define the \emph{level number} of $G$, denoted $\ell(G)$, to be the minimum level number of $G$ ranging over all embeddings of $G$.
\end{definition}
It is immediate to see that the level number of a graph $G$ is a lower bound for the level number of each of its embeddings.
\begin{remark}\label{rmk: level number 1 iff spherical}
    A leveled spatial graph $\G$ has level number~1 if and only if it is trivial. A graph $G$ with at least one cycle has level number~1 if and only if it is a planar graph.

    We show the statement about spatial graphs, as the one about graphs follows from it by the definition of the level number. A leveled embedding with one level embeds in the sphere. For the other direction of this statement, any embedding of a spatial graph in the sphere is a one-level leveled embedding, considering a cycle contained in the boundary of any face as a spine for the spatial graph. Thus any trivial spatial graph has level number one.
\end{remark}
Since Hamiltonian leveled spatial graphs play an important role in studying leveled spatial graphs, it is natural to formulate~\cref{def: level number spatial graph} for a Hamiltonian leveled spatial graph.
\begin{definition}
    Given a Hamiltonian spatial graph $\G$, we define its \emph{Hamiltonian level number}, denoted $h\ell(\G)$, to be the minimum number of levels of $\G$ ranging over all choices of Hamiltonian spines for a Hamiltonian leveled embedding of $\G$. The Hamiltonian level number $\hl(G)$ of an abstract Hamiltonian graph $G$ is the minimum Hamiltonian level number of all its embeddings.
\end{definition}
Note that, due to~\cref{cor: Hamiltonian graphs have a leveled embedding}, any Hamiltonian graph $G$ has finite Hamiltonian level number.

Equivalently, the Hamiltonian level number of a Hamiltonian graph $G$ is equal to the smallest chromatic number of the conflict graph of $G$ with respect to any Hamiltonian cycle of $G$.
%If the Hamiltonian graph $G$ has all vertices of degree 3, then its Hamiltonian level number is exactly the chromatic number of its circle graph.

The level number of a Hamiltonian graph is a lower bound for its Hamiltonian level number by definition. However, the following example shows that these numbers may differ from each other.
\begin{example}
    $\ell(K_4)=1$ but $\hl(K_4)=2$, as is illustrated in \Cref{fig: K4 in one and two levels}.%, while $h\ell(K_5)>2,$ because $K_5$ is abstractly non-planar. Indeed, if the Hamiltonian level number of $K_5$ was 2, then it would be possible to embed the Hamiltonian leveled embedding with two levels of $K_5$ on a sphere, setting the equator to be the Hamiltonian spine and the two levels to be embedded in the upper and lower hemisphere respectively. It is not hard to prove that $h\ell(K_5)=3$.
\end{example}
\begin{figure}
    \centering
    \includegraphics{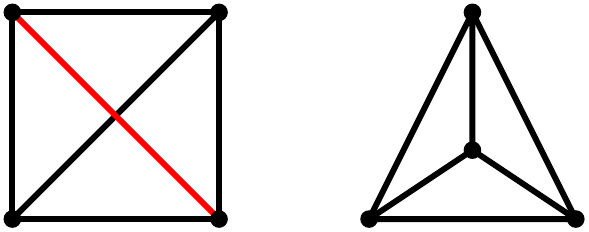}
    \caption{The Hamiltonian level number of $K_4$ differs from its level number: $\ell(K_4)=1$ and $\hl(K_4)=2$.}
    \label{fig: K4 in one and two levels}
\end{figure}
We adapt \cref{rmk: level number 1 iff spherical} for the Hamiltonian level number.
\begin{proposition}
    A Hamiltonian graph $G$ is planar if and only if its Hamiltonian level number is at most 2. Moreover, $G$ is outerplanar if and only if its Hamiltonian level number is 1.
\end{proposition}
\begin{proof}
    If $G$ is planar, then it has an embedding on a sphere $S^2$. This embedding is Hamiltonian leveled with at most 2 levels by taking a Hamiltonian cycle of $G$ as the spine. Conversely, any Hamiltonian leveled embedding of $G$ with at most two levels is an embedding of $G$ on a sphere, and therefore trivial, implying that $G$ is planar.

    $G$ has a Hamiltonian leveled embedding with one level if and only if it can be embedded in a closed disk with all vertices lying on the boundary of the disk. This is an outerplanar embedding of $G$.
\end{proof}

Note that both the level number and the Hamiltonian level number are not increasing, in the sense specified by the following example.
\begin{example}\label{exm: level number is not increasing}
    Consider the Hamiltonian graph $K_{4,4}$. Both its level number $\ell(K_{4,4})$ and its Hamiltonian level number $\hl(K_{4,4})$ equal 4, as we show in \cref{thm: level number of Knn}. Consider instead the Hamiltonian graph $G$ obtained from $K_{4,4}$ by adding the four dashed edges in \Cref{fig: level number is not increasing}. One might expect $\hl(G) \ge \hl(K_{4,4})$ as $K_{4,4}$ is a subgraph of $G$ obtained by removing some edges. However, $\hl(G) = 3$ as shown in \Cref{fig: level number is not increasing}, and since $\ell(G)\le  \hl(G)$, it also holds that $\ell(G) \le \ell(K_{4,4})$.
\end{example}
\begin{figure}
    \centering
    \includegraphics[width=0.5\linewidth]{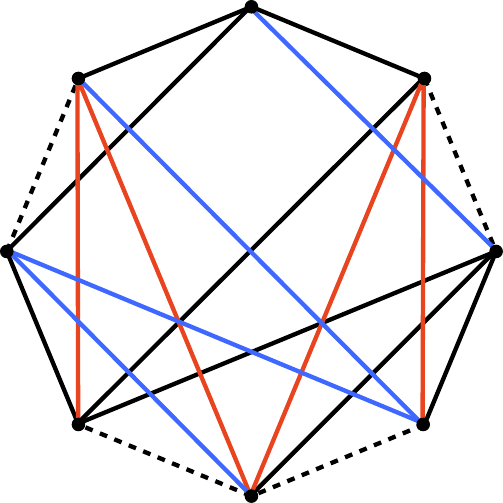}
    \caption{A Hamiltonian graph with Hamiltonian level number 3. The graph contains $K_{4,4}$ as subgraph removing the four dashed edges.}
    \label{fig: level number is not increasing}
\end{figure}
\begin{comment}
Note that the Hamiltonian level number is not increasing, in the sense specified by the following example.
\begin{example}\label{exm: hamilton level number is not increasing}
    Let $H$ be the Hamiltonian graph pictured in \Cref{fig: Hamiltonian level number is not increasing in supergraphs}. The Hamiltonian level number of $H$ is $4$ as it has a unique Hamiltonian cycle. Consider $G=M_{16}$, the M\"obius ladder on 16 vertices, which is the Hamiltonian graph obtained from $H$ by adding four edges, pictured left in \Cref{fig: mobius ladder}. One might expect $\hl(G) \ge \hl(H)$ as $H$ is a subgraph of $G$ obtained by removing edges. However, $\hl(G)=3$ since a Hamiltonian leveled embedding with three levels can be given as shown right in \Cref{fig: mobius ladder}, and the non-planarity of $G$ implies that $\hl(G)>2$.
\end{example}
\begin{figure}
    \centering
    \includegraphics[width=0.5\textwidth]{HamiltonLevelNumberDecreases.pdf}
    \caption{A Hamiltonian leveled spatial graph with Hamiltonian level number 4.}
    \label{fig: Hamiltonian level number is not increasing in supergraphs}
\end{figure}
\end{comment}

For complete graphs, we show the following relation between level numbers and Hamiltonian level numbers.
\begin{proposition}\label{prop: level number of Kn is at most as big as Hamiltonian level of Kn-1 plus 1}
    For $n \ge 4$, $\ell(K_n) \le \hl(K_{n-1}) +1.$
\end{proposition}
\begin{proof}
    Consider a Hamiltonian leveled embedding of $K_{n-1}$ which realizes the Hamiltonian level number $\hl(K_{n-1})$ and add one vertex to its spine. The edges between the added vertex and all other vertices can be all assigned to a single new level. Therefore, the resulting spatial graph is a leveled embedding of $K_n$ with exactly one more level than the embedding of $K_{n-1}$ we started with. %Given that all existing edges of the embedding of $K_{n-1}$ have a crossing with one of the added edges, the constructed leveled embedding of $K_n$ has exactly one more level than the embedding of $K_{n-1}$ we started with.
    %formed by a vertex adjacent to each vertex of $K_{n-1}$. This is a leveled embedding of $K_n$.
\end{proof}
Note that the inequality in \cref{prop: level number of Kn is at most as big as Hamiltonian level of Kn-1 plus 1} for $n\ge 5$ is sharp if exactly one vertex does not belong to the spine of the leveled embedding of $K_n$. Indeed, the minimal number of levels to embed $K_{n-1}$ in such a way that all vertices lie on the spine is $\hl(K_{n-1})$ by definition. The vertex not belonging to the spine needs to be assigned to a different level, because $n \ge 5$ and so the vertex needs to be connected to at least four vertices on the spine.

\subsection{Comparison to existing similar invariants}
%In \cite{LeveledCellularEmbeddings} was mentioned that level embeddings are related to book embeddings. We would like now to make the relationship between leveled embeddings and other existing invariants clear.
%There are two concepts of thickness defined in literature for a graph that we want to relate to the level number: graph-theoretical thickness (\cite{TutteThicknessofGraphs}) and book thickness (\cite{BookThickness}).
%These invariants have in common that, taken a graph, they split it in a minimal collection of certain planar subgraphs.
Similar to the level number, notions of thickness are measures of the non-planarity of a graph.
The \emph{graph-theoretical thickness} of a graph~$G$, denoted $\theta(G)$, is the minimum number of planar subgraphs whose union is $G$. The \emph{book thickness} of $G$, denoted $\bt(G)$, is the smallest number $n$ such that $G$ has an $n$-book embedding.

%In the definition of graph theoretical thickness it is equivalent to ask that all planar subgraphs whose union is $G$ have fixed position in 3-space, similarly to what happens to the planar subgraphs in the book thickness or the level number of $G$. This is due to Kainen~\cite{RealAndLinearThickness}.

\begin{proposition}\label{prop: inequalities concerning thickness level number and book thickness}
    For any abstract graph $G$ and any Hamiltonian graph $H$ \begin{enumerate}
        \item $\theta(G) \le \ell(G),$
        \item $\bt(H) \le \hl(H).$
    \end{enumerate}
\end{proposition}
\begin{proof}
    \hfill
    \begin{enumerate}
        \item The levels of a leveled embedding are planar subgraphs whose union is the whole graph.
        \item A Hamiltonian leveled embedding of $H$ with $n$ levels is an $n$-book embedding of $H$.
    \end{enumerate}
\end{proof}
In general, neither of these inequalities is an equality, even though they are both sharp for $G=H=K_4$. For example, $\theta(K_6)=2$ (see Beineke and Harary~\cite{HararyThicknessOfCompleteGraph}) but $\ell(K_6)=3$, as we show in \cref{prop: level number of Kn}. Also, $K_{4,4}$ has book thickness equal to 3, as shown by Bernhart and Kainen~\cite{BookThickness}, while $\hl(K_{4,4})=4$ as we show in \cref{thm: level number of Knn}. %the graph $H$ of \cref{exm: hamilton level number is not increasing} has book thickness 3 and Hamiltonian level number 4. 
In the second inequality of \cref{prop: inequalities concerning thickness level number and book thickness}, it is not possible to exchange the Hamiltonian level number with the level number. In other words, it is not true that $\bt(G) \le \ell(G)$ for a graph $G$. This can be seen as follows: the level number of any planar graph is 1 by \cref{rmk: level number 1 iff spherical}, whereas  Bernhart and Kainen showed that the book thickness of non-Hamiltonian maximal planar graphs is at least 3~(\cite{BookThickness}). 

%As it is custom in literature (\cite{HararyThicknessOfCompleteGraph}, \cite{ThicknessOfArbitraryCompleteGraph}, \cite{K16hasthickness3}, \cite{BookThickness}), we focus on the level number of complete graphs.
For complete graphs with at least five vertices, the level number, Hamiltonian level number, and book-thickness all coincide:
\begin{proposition}\label{prop: level number of Kn}
    For any $n\ge 5$, $\hl(K_n)= \ell(K_n)=\bt(K_n) = \lceil \frac{n}{2}\rceil$.
\end{proposition}
\begin{proof}
    The last equality is due to Bernhart and Kainen~\cite{BookThickness} and it holds for $n\ge 4$. 
    Every $\lceil \frac{n}{2}\rceil$-book embedding of $K_n$ has a Hamiltonian printing cycle. Because the printing cycle can be chosen as spine and the levels are inherited from the pages, it is a Hamiltonian leveled embedding with number of levels equal to the number of pages of the book embedding. This, together with 2 of \cref{prop: inequalities concerning thickness level number and book thickness}, ensures that $\hl(K_n)=\bt(K_n).$

    Lastly, we show that $\hl(K_n)=\ell(K_n).$ Since $\ell(G) \leq \hl(G)$, it is sufficient to show that considering a non-Hamiltonian spine for $K_n$ does not yield a level number smaller than $\lceil \frac{n}{2}\rceil$.
    If the spine was a non-Hamiltonian cycle and there were two or more vertices of $K_n$ not contained in this cycle, the vertices that are not part of the cycle must belong to the same level given that they are adjacent. In that case, both vertices $v$ and $w$ and all edges issuing from them need to be embedded in a disk $D$ whose boundary coincides with the spine. In $D$ we can first place $v$ and connect it with all vertices on the spine, decomposing $D$ into triangles. Then, we would need to place $w$ into a triangle, meaning that we can connect it with at most three vertices. If $n\ge5$, $w$ has degree at least 4 and therefore $w$ and $v$ cannot belong to the same level. Therefore, we are left to consider spines with at least $n-1$ vertices. If exactly one vertex does not belong to the spine, we obtain the leveled embedding described in the proof of \cref{prop: level number of Kn is at most as big as Hamiltonian level of Kn-1 plus 1}. As for this embedding the inequality of \cref{prop: level number of Kn is at most as big as Hamiltonian level of Kn-1 plus 1} is sharp, its level number equals $\hl(K_{n-1})+1$.  As $\hl(K_{n-1})+1 = \lceil\frac{n-1}{2}\rceil +1 \ge \lceil \frac{n}{2}\rceil$ for $n\ge 5$, the statement follows.% As $n\ge 6$, no leveled embedding exist in this case. We now proceed by induction on $n$. For $n=6$, $\hl(K_6)=3$, which is also the number of levels we obtain by having exactly one vertex not belong to the spine, due to \cref{prop: level number of Kn is at most as big as Hamiltonian level of Kn-1 plus 1}. If two or more vertices belong to the same level but not to the spine, then $K_6$ (and also $K_n$
\end{proof}
Computing the (Hamiltonian) level numbers for complete graphs $K_n$ on few vertices directly, we get the following:
\begin{corollary}\label{cor: formula for level number of Kn}
    $$\hl(K_n)= \begin{cases}
        0, & \text{if}\ n=3\\
        \lceil \frac{n}{2}\rceil, & \text{if}\ n \ge 4.
    \end{cases}$$
    $$\ell(K_n) = \begin{cases}
     \infty, & \text{if}\ n<3\\
        0, & \text{if}\ n=3\\
        1, & \text{if}\ n=4\\
        \lceil \frac{n}{2}\rceil, & \text{if}\ n \ge 5. 
    \end{cases}$$
\end{corollary}
Even though the level number or the Hamiltonian level number of a graph is not necessarily smaller than that of a supergraph, as seen in \cref{exm: level number is not increasing}, we still have an upper bound in the extremal case where the supergraph is the complete graph:
\begin{proposition}
    Let $G$ be a Hamiltonian graph with $n\ge5$ vertices. Then $\ell(G) \le \hl(G) \le \hl(K_n)=\ell(K_n).$
\end{proposition}
\begin{proof}
    The first inequality is a direct consequence of the definitions while the last equality is given in \cref{prop: level number of Kn}. We are left to prove that $\hl(G) \le \hl(K_n)$. Consider a spine which achieves the Hamiltonian level number of $G$, with vertices $v_1, \dots, v_n$ of $G$ in a cyclic ordering. As $K_n$ also has $n$ vertices, rename the vertices on a spine which achieves the Hamiltonian level number for $K_n$ as $v_1, \dots, v_n$ in the same cyclic ordering as for $G$. The Hamiltonian level number is equal to the minimum of the chromatic numbers of the conflict graphs ranging over all choices of spines. Fixing a spine, the chromatic number of the conflict graph does not decrease if edges are added to $G$ to obtain the complete graph $K_n$. The statement is therefore proven.
\end{proof}
 To prove the following closed formula for the (Hamiltonian) level number of complete bipartite graphs, we use
the classical result that all cycles in a bipartite graph have even length~\cite{MR1367739}.
\begin{theorem}\label{thm: level number of Knn}\hfill \begin{itemize}
    \item $\ell(K_{n,n}) = \hl(K_{n,n}) = n$ for $n \ge 3$;
    \item $\ell(K_{m,n}) = m$ for $m>n>2.$
    \item $\ell(K_{m,2})=1$ for any $m$.
\end{itemize}
\end{theorem}
\begin{proof} \hfill
    \begin{itemize}
        \item Let $A=\{a_1, \dots, a_n\}$ and $B=\{b_1, \dots, b_n\}$ be the two maximal subsets of the $2n$ vertices of $K_{n,n}$ that are not connected by edges. We first prove that $\ell(K_{n,n}) = \hl(K_{n,n}).$ Consider a non-Hamiltonian cycle~$C$ of $K_{n,n}$. %As every cycle in $K_{n,n}$ is even and $K_{n,n}$ has an even number of vertices, the number of vertices that do not belong to~$C$ is also even. 
    Since $C$ alternates between vertices of $A$ and vertices of $B$, the vertices that do not belong to~$C$, are also equally divided between $A$ and $B$. For $K_{n,n}$ to be a leveled embedding with respect to spine $C$, all vertices not belonging to~$C$ must belong to the same level, since a vertex $a_i$ is adjacent to every vertex in $B$. Furthermore, vertices that do not belong to $C$ as well as all edges issuing from them must be embedded in a disk $D$ whose boundary coincides with $C$. Place one vertex that does not belong to~$C$, say $a_1$ without loss of generality, on $D$ and connect it to all vertices of $B$ that lie on~$C$, dividing $D$ into 4-gons. There is a second vertex not belonging to~$C$ which belongs to $B$, say $b_1$ without loss of generality, and since the embedding must contain all vertices not belonging to $C$ in one level, it must be placed in one of the 4-gons created by placing $a_1$. Therefore $b_1$ can be connected to at most two vertices of $A$. But $n\ge 3$, so $b_1$ is connected to at least 3 other vertices. Therefore, it is impossible to find a leveled embedding of $K_{n,n}$ with a non-Hamiltonian spine, and consequently their level number equals their Hamiltonian level number.  %that some vertices of $K_{n,n}$ do not belong to the spine.%all vertices of $K_{n,n}$ need to belong to the spine of a leveled embedding of $K_{n,n}$.

    We now prove that $\hl(K_{n,n})=n.$ Consider the vertices of $K_{n,n}$ to be on a cycle and label them $x_1, x_2, \dots, x_{2n}$, such that $x_i\in A$ if $i$ is odd and $x_i\in B$ if $i$ is even. All edges issuing from vertex $x_{2i+1} \in A$ can be assigned to the same level, and this holds for $i=0,\dots, n-1$, implying that $\hl(K_{n,n})\le n.$ We are left to show that $n$ levels are needed.
    If $n$ is odd, the edges $(x_1, x_{n+1}), (x_3, x_{n+3}), \dots, (x_{2n-1}, x_{n-1})$ all conflict with each other. Therefore, there must be at least $n$ levels.
    If $n$ is even, the edges $(x_1, x_{n}), (x_2, x_{n+1}), \dots, (x_{n-1}, x_{2n-2})$ all conflict with each other. Therefore, there must be at least $n-1$ levels. Assign edge $(x_i, x_{i+n-1})$ to level $l_i$ for $i=1, \dots, n-1$. The edge $e=(x_n, x_{2n-1})$ conflicts with the $n-1$ previous edges $(x_2, x_{n+1}), \dots, (x_{n-1}, x_{2n-2})$ but not with $(x_1, x_n)$. If we assign $e$ to level $l_n$, the proof is complete because $K_{n,n}$ has at least $n$ levels. Otherwise, we have to assign $e$ to level $l_1$. The edge $(x_{n+1}, x_{2n})$ is assigned to level $l_2$ since it conflicts with $(x_3, x_{n+2}), \dots, (x_{n}, x_{2n-1})$. We continue in this fashion, assigning edge $(x_{n+t}, x_{2n+t-1})$ to level $l_{t+1}$ for $t=0, \dots, n-2$, where the indices need to be considered modulo $2n$. Consider now the edge $(x_{2n-1}, x_{3n-2})= (x_{2n-1}, x_{n-2}).$ It conflicts with the $n-1$ edges $(x_{n+1}, x_{2n}), \dots, (x_{2n-2}, x_{3n-3})=(x_{2n-2}, x_{n-3})$, which are assigned to levels $l_2, \dots, l_{n-1}$. It also conflicts with edge $(x_1, x_n)$, assigned to level $l_1$. Therefore, at least $n$ levels are necessary.
    \item If $m> n$, $K_{m,n}$ has no Hamiltonian cycle. If $n>2$, an argument similar to the one showing that $\ell(K_{n,n})=\hl(K_{n,n})$ ensures that the $m-n$ vertices that do not belong to the spine must be assigned to different levels, and that no edge other than those issuing from each of those vertices can be assigned to the same level. Therefore, $\ell(K_{m,n})= (m-n) + \hl(K_{n,n})= m$.
    \item $K_{m,2}$ is a planar graph for any $m$, so its level number is equal to one by \cref{rmk: level number 1 iff spherical}.
    \end{itemize}
\end{proof}
\section{Outlook}
Proposition~\ref{prop: level number of Kn} and~\cref{thm: level number of Knn} characterize the level number and the Hamiltonian level number of complete graphs and complete bipartite graphs. Characterizing the (Hamiltonian) level number of other families of graphs is relevant for their cellular embedding possibilities. In particular, characterizing the spatial graphs that have a leveled embedding with level number smaller or equal than 4 would determine  spatial graphs that are guaranteed to admit cellular embeddings, due to Proposition~1 of \cite{LeveledCellularEmbeddings}.

When focusing on determining the Hamiltonian level number of a Hamiltonian graph, we can assume without loss of generality that all vertices of the Hamiltonian graph have degree three, i.e. the graph is cubic. Indeed, by Lemma~1 of \cite{LeveledCellularEmbeddings}, there exists a splitting of vertices that keeps the conflict graph unchanged. Since the number of levels of a leveled Hamiltonian spatial graph is given by the chromatic number of the conflict graph of the underlying abstract graph, the Hamiltonian level number is preserved as well.

The conflict graphs of cubic Hamiltonian graphs are known as circle graphs \cite{MR1367739}.
Therefore, all results on chromatic numbers of circle graphs directly translate to the Hamiltonian level number of a graph as upper bounds. For a Hamiltonian leveled graph~$\mathcal{H}$ with abstract graph~$H$, the chromatic number of the circle graph of $H$ equals the minimal level number of all leveled embeddings of $H$ that have the same spine. Thus the only difference between the chromatic number of the circle graph of $H$ and the Hamiltonian level number of $H$ is that the latter minimizes the chromatic numbers over all choices of Hamiltonian cycles as spine.

Therefore, studying the chromatic number of circle graphs helps determining the Hamiltonian level number of Hamiltonian graphs. However, determining the chromatic number of arbitrary circle graphs is NP-hard, as proved by Garey, Johnson, Miller and Papadimitriou~\cite{ChromaticCircleNPhard}.

\section{Acknowledgement}
The authors want to thank Riya Dogra for helpful discussions.

\section{Bibliography}
\bibliographystyle{alpha}
\bibliography{bibliography}
\end{document}